\documentclass[10pt,oneside,leqno]{amsart}

\usepackage{amsmath,amsfonts,amssymb}
\usepackage{xypic}
\usepackage[utf8x]{inputenc}
\usepackage[english]{babel}

\newtheorem{thm}{Theorem}[section]
\newtheorem{prop}[thm]{Proposition}
\newtheorem{cor}[thm]{Corollary}
\newtheorem{lemma}[thm]{Lemma}
\theoremstyle{definition}
\newtheorem{defi}[thm]{Definition}
\newtheorem{rem}[thm]{Remark}
\newtheorem{ex}[thm]{Example}

\newcommand{\N}{\mathbb{N}}

\newcommand{\R}{\mathbb{R}}

\newcommand{\st}{\;:\;}

\newcommand{\correnti}{D}

\newcommand{\duale}[1] {{{#1}^{*}}}

\newcommand{\sspace}{\cdot}
\newcommand{\ssspace}{\cdot\cdot}

\newcommand{\formepm}[2]{\wedge^{#1,\,#2}_{+\,-}\,}
\newcommand{\correntipm}[2]{D_{#1,\,#2}^{+\,-}\,}
\newcommand{\correntipmalto}[2]{D^{#1,\,#2}_{+\,-}\,}

\DeclareMathOperator{\imm}{im}
\DeclareMathOperator{\de}{d}
\DeclareMathOperator{\id}{Id}
\DeclareMathOperator{\Vol}{Vol}

\DeclareMathOperator{\End}{End}

\DeclareMathOperator{\rk}{rk}

\newcommand{\Cpf}{$\mathcal{C}^\infty$-pure-and-full}

\newcommand{\Cf}{$\mathcal{C}^\infty$-full}
\newcommand{\Cp}{$\mathcal{C}^\infty$-pure}
\newcommand{\pf}{pure-and-full}

\newcommand{\f}{full}
\newcommand{\p}{pure}
\newcommand{\del}{\partial}

\newcommand{\kal}{K\"{a}hler}

\newcommand{\proj}{\pi}

\newcommand{\D}{$\mathbf{D}$}
\newcommand{\para}{\D}

\title[Cohomology of \para-complex manifolds]
{Cohomology of \para-complex manifolds}
\author{Daniele Angella}
\address{Dipartimento di Matematica ``Leonida Tonelli''\\
Universit\`{a} di Pisa \\
Largo Bruno Pontecorvo 5, 56127\\ 
Pisa, Italy}
\email{angella@mail.dm.unipi.it}
\author{Federico Alberto Rossi}
\address{Dipartimento di Matematica e Applicazioni\\
Universit\`{a} di Milano Bicocca\\
Via Cozzi 53, 20125 \\
Milano, Italy} 
\email{f.rossi46@campus.unimib.it}

\keywords{$\mathcal{C}^\infty$-pure-and-full structure; para-complex structure; $\mathbf{D}$-complex structure; $\mathbf{D}$-K\"ahler; nilmanifold; cohomology; deformation}
\thanks{The authors are partially supported by GNSAGA of INdAM}
\subjclass[2010]{53C15; 57T15; 32G07}
\begin{document}

\begin{abstract}
 In order to look for a well-behaved counterpart to Dolbeault cohomology in \para-complex geometry, we study the de Rham cohomology of an almost \para-complex manifold and its subgroups made up of the classes admitting invariant, respectively anti-invariant, representatives with respect to the almost \para-complex structure, miming the theory introduced by T.-J. Li and W. Zhang in \cite{li-zhang} for almost complex manifolds. In particular, we prove that, on a $4$-dimensional \para-complex nilmanifold, such subgroups provide a decomposition at the level of the real second de Rham cohomology group. Moreover, we study deformations of \para-complex structures, showing in particular that admitting \para-K\"ahler structures is not a stable property under small deformations.
\end{abstract}

\maketitle

\section*{Introduction}\label{sec:introduction}

\para-complex geometry arises naturally as a counterpart of complex geometry. Indeed, an \emph{almost \para-complex} structure (also called \emph{almost $\mathbf{D}$-structure}) on a manifold $X$ is an endomorphism $K$ of the tangent bundle $TX$ whose square $K^2$ is equal to the identity on $TX$ and such that the rank of the two eigenbundles corresponding to the eigenvalues $\left\{-1,\,1\right\}$ of $K$ are equi-dimensional (for an \emph{almost complex} structure $J\in\End(TX)$, one requires just that $J^2=-\id_{TX}$). Furthermore, many connections with other branches in Mathematics and Physics (in particular, with product structures, bi-Lagrangian geometry and optimal transport problem: see, e.g., \cite{kim-mccann-warren, harvey-lawson, cruceanu-fortuny-gadea} and the references therein) have recently been achieved.

A problem dealing with \emph{\para-complex} structures (that is, almost \para-complex structures satisfying an integrability condition, expressed in terms of their Nijenhuis tensor) is that one has to handle with semi-Riemannian metrics and not with Riemannian ones. In particular, one can try to reformulate a \emph{\para-Dolbeault cohomological theory} for \para-complex structures, in the same vein as Dolbeault cohomology theory for complex manifolds: but one suddenly finds that such \para-Dolbeault groups are in general not finite-dimensional (for example, yet the space of \para-holomorphic functions on the product of two equi-dimensional manifolds is not finite-dimensional). In fact, one loses the ellipticity of the second-order differential operator associated to such \para-Dolbeault cohomology. Therefore, it would be interesting to find some other (well-behaved) counterpart to \para-Dolbeault cohomology groups.

Recently, T.-J. Li and W. Zhang considered in \cite{li-zhang} some subgroups, called $H^+_J(X;\R)$ and $H^-_J(X;\R)$, of the real second de Rham cohomology group $H^2_{dR}(X;\R)$ of an almost complex manifold $\left(X,\,J\right)$, characterized by the type of their representatives with respect to the almost complex structure (more precisely, $H^+_J(X;\R)$, respectively $H^-_J(X;\R)$, contains the de Rham cohomology classes admitting a $J$-invariant, respectively $J$-anti-invariant, representative): in a sense, these subgroups behave as a ``generalization'' of the Dolbeault cohomology groups to the complex non-K\"ahler and non-integrable cases. In particular, these subgroups seem to be very interesting for $4$-dimensional compact almost complex manifolds and in studying relations between cones of metric structures, see \cite{li-zhang, draghici-li-zhang, angella-tomassini-2}. In fact, T. Dr\v{a}ghici, T.-J. Li and W. Zhang proved in \cite[Theorem 2.3]{draghici-li-zhang} that every almost complex structure $J$ on a $4$-dimensional compact manifold $X$ induces the decomposition $H^2_{dR}(X;\R)=H^+_J(X;\R)\oplus H^-_J(X;\R)$; the same decomposition holds true also for compact K\"ahler manifolds, thanks to Hodge decomposition (see, e.g., \cite{li-zhang}), while examples of complex and almost complex structures in dimension greater than $4$ for which it does not hold are known.

In this work, we reformulate T.-J. Li and W. Zhang's theory in the almost \para-complex case, constructing subgroups of the de Rham cohomology linked with the almost \para-complex structure. In particular, we are interested in studying when an almost \para-complex structure $K$ on a manifold $X$ induces the cohomological decomposition
$$ H^2_{dR}(X;\R) \;=\; H^{2\,+}_K(X;\R) \oplus H^{2\,-}_K(X;\R) $$
through the \para-complex subgroups $H^{2\,+}_K(X;\R)$, $H^{2\,-}_K(X;\R)$ of $H^2_{dR}(X;\R)$, made up of the classes admitting a $K$-invariant, respectively $K$-anti-invariant representative;
such almost \para-complex structures will be called \emph{\Cpf\ (at the $2$-nd stage)}, miming T.-J. Li and W. Zhang's notation in \cite{li-zhang}.

We prove some results and provide some examples showing that the situation, in the (almost) \para-complex case, is very different from the (almost) complex case.\\
In particular, in Example \ref{es 2.5} and in Example \ref{es 2.6}, we show that compact \emph{\para-K\"ahler} manifolds (that is, \para-complex manifolds endowed with a symplectic form that is anti-invariant under the action of the \para-complex structure) need not to satisfy the cohomological decomposition through the \para-complex subgroups we have introduced. Furthermore, Example \ref{es 2.8} shows a $4$-dimensional almost \para-complex nilmanifold that does not satisfy such a \para-complex cohomological decomposition, providing a difference with \cite[Theorem 2.3]{draghici-li-zhang} by T. Dr\v{a}ghici, T.-J. Li and W. Zhang.

With the aim to write a partial counterpart of \cite[Theorem 2.3]{draghici-li-zhang} in the \para-complex case, we prove the following result (see \S\ref{sec:definitions} for the definition of \emph{\pf} property).

\smallskip
\noindent {\bfseries Theorem \ref{thm:4-dim}.}
{\itshape
 Every invariant \para-complex structure on a $4$-dimensional nilmanifold is \Cpf\ at the $2$-nd stage and hence also \pf\ at the $2$-nd stage. 
}
\smallskip

\noindent Furthermore, we prove that the hypotheses on integrability, nilpotency and dimension can not be dropped out.

Lastly, we study explicit examples of deformations of \para-complex structures (see \cite{medori-tomassini, rossi} for a general account). In particular, we provide examples showing that the dimensions of the \para-complex subgroups of the cohomology can jump along a curve of \para-complex structures. Furthermore, we prove the following result, which provides another strong difference with respect to the complex case (indeed, recall that the property of admitting a K\"ahler metric is stable under small deformations of the complex structure, as proved by K. Kodaira and D.~C. Spencer in \cite{kodaira-spencer-3}).

\smallskip
\noindent {\bfseries Theorem \ref{thm:para-kahler-deformations}.}
{\itshape
 The property of being \para-\kal\ is not stable under small deformations of the \para-complex structure.
}

\medskip

The paper is organized as follows.\\
In Section \ref{sec:definitions}, we recall what a \para-complex structure is, we introduce the problem of studying \para-complex subgroups of cohomology and we introduce the concept of \Cpf\ \para-complex structure to mean a structure inducing a \para-complex decomposition in cohomology. We prove that every manifold given by the product of two equi-dimensional differentiable manifolds is \Cpf\ with respect to the natural \para-complex structure (see Theorem \ref{thm:products}).\\
In Section \ref{sec:invariant-cpf-para-complex-structures-on-solvmanifolds}, we introduce analogous definitions at the linear level of the Lie algebra associated to a (quotient of a) Lie group. In particular, we prove that, for a completely-solvable solvmanifold with an invariant \para-complex structure, the problem of the existence of a \para-complex cohomological decomposition reduces to such a (linear) decomposition at the level of its Lie algebra (see Proposition \ref{prop:linear-cpf-invariant-cpf}).\\
In Section \ref{sec:cpfness-solvmanifolds}, we prove Theorem \ref{thm:4-dim} and we give some examples showing that the hypotheses we assume can not be dropped out. Moreover, we provide examples showing that admitting \para-K\"ahler structures does not imply being \Cpf\ (see Proposition \ref{prop:para-kahler}).\\
In Section \ref{sec:deformations}, we study deformations of \para-complex structures, providing an example to prove Theorem \ref{thm:para-kahler-deformations} and showing that, in general, jumping for the dimensions of the \para-complex subgroups of cohomology can occur.

\medskip

\noindent{\itshape Acknowledgments.}
The authors would like to warmly thank Adriano Tomassini and Costantino Medori for their encouragement, their constant support and for many interesting conversations and useful remarks. The first author would like to thank also Serena Guarino Lo Bianco for a motivating conversation.

\section{\para-complex decompositions of cohomology and homology}\label{sec:definitions}
Let $X$ be a $2n$-dimensional compact manifold. Consider $K\in\End(TX)$ such that $K^2=\lambda\,\id_{TX}$ where $\lambda\in\left\{-1,\,1\right\}$: if $\lambda=-1$, we call $K$ an \emph{almost complex} structure; if $\lambda=1$, one gets that $K$ has eigenvalues $\left\{1,\,-1\right\}$ and hence there is a decomposition $TX=T^+X\oplus T^-X$ where $T^\pm X$ is given, point by point, by the eigenspace of $K$ corresponding to the eigenvalue $\pm 1$, where $\pm\in\left\{+,\,-\right\}$.

\noindent Recall that an \emph{almost \para-complex} structure (also called \emph{almost para-complex structure}) on $X$ is an endomorphism $K\in\End(TX)$ such that
$$ K^2\;=\;\id_{TX} \qquad \text{ and } \qquad \rk T^+X\;=\;\rk T^-X\;=\;\frac{1}{2}\dim X \;.$$
An almost \para-complex structure is said \emph{integrable} (and hence is called \emph{\para-complex}, or also \emph{para-complex}) if moreover
$$ \left[T^+X,\,T^+X\right] \;\subseteq\; T^+X \qquad \text{ and } \qquad \left[T^-X,\,T^-X\right] \;\subseteq\; T^-X \;,$$
(or, equivalently, if the \emph{Nijenhius tensor} of $K$, defined by
$$ N_K(\sspace,\ssspace) \;:=\; \left[\sspace,\,\ssspace\right]+\left[K\sspace,\,K\ssspace\right]-K\left[K\sspace,\,\ssspace\right]-K\left[\sspace,\,K\ssspace\right] \;, $$
vanishes).
We refer, e.g., to \cite{harvey-lawson, alekseevsky-medori-tomassini, CMMS04 Spec geom IeII, cruceanu-fortuny-gadea, andrada-barberis-dotti-ovando, andrada-salamon, rossi, rossi-2} and the references therein for more results about (almost) \para-complex structures and motivations for their study.

Note that, in a natural way, starting from $K\in\End(TX)$, one can define an endomorphism $K\in\End(T^*X)$ and hence a natural decomposition $T^*X=T^{*\,+}X\oplus T^{*\,-}X$ into the corresponding eigenbundles. Therefore, for any $\ell\in\N$, on the space of $\ell$-forms on $X$, we have the decomposition
\begin{eqnarray*}
\wedge^\ell X &:=& \wedge^\ell \left(T^*X\right) \;=\; \wedge^\ell \left(T^{*\,+}X\oplus T^{*\,-}X\right) \\[5pt]
&=& \bigoplus_{p+q=\ell} \wedge^p\left(T^{*\,+}X\right)\otimes\wedge^q\left(T^{*\,-}X\right) \;=:\; \bigoplus_{p+q=\ell} \formepm{p}{q}X
\end{eqnarray*}
where, for any $p,q\in\N$, the natural extension of $K$ on $\wedge^\bullet X$ acts on $\formepm{p}{q}X:=\wedge^p\left(T^{*\,+}X\right)\otimes\wedge^q\left(T^{*\,-}X\right)$ as $\left(+1\right)^p\left(-1\right)^q\id$. In particular, for any $\ell\in\N$,
$$ \wedge^\ell X \;=\; \underbrace{\bigoplus_{p+q=\ell,\;q\text{ even}} \formepm{p}{q}X}_{=:\; \wedge^{\ell\,+}_KX} \,\oplus\, \underbrace{\bigoplus_{p+q=\ell,\;q\text{ odd}} \formepm{p}{q}X}_{=:\; \wedge^{\ell\,-}_KX} $$
where
$$ K\lfloor_{\wedge^{\ell\,+}_KX} \;=\; \id \qquad \text{ and } \qquad K\lfloor_{\wedge^{\ell\,-}_KX} \;=\; -\id \;.$$

If an (integrable) \para-complex structure $K$ is given, then the exterior differential splits as
$$ \de \;=\; \del_+ + \del_- $$
where
$$ \del_+\;:=\;\proj_{\formepm{p+1}{q}}\colon \formepm{p}{q}\to\formepm{p+1}{q} $$
and
$$ \del_-\;:=\;\proj_{\formepm{p}{q+1}}\colon \formepm{p}{q}\to\formepm{p}{q+1} \;,$$
$\proj_{\formepm{r}{s}}\colon \formepm{\bullet}{\bullet}\to \formepm{r}{s}$ being the natural projection.
In particular, the condition $\de^2=0$ gives
$$
\left\{
\begin{array}{rcl}
 \del_+^2 &=& 0 \\[5pt]
 \del_+\del_-+\del_-\del_+ &=& 0 \\[5pt]
 \del_-^2 &=& 0
\end{array}
\right.
$$
and hence one could define the \emph{\para-Dolbeault cohomology} as
$$ H^{\bullet,\bullet}_{\del_+}(X;\R) \;:=\; \frac{\ker\del_+}{\imm\del_+} \;,$$
see \cite{krahe}. Unfortunately, one can not hope to adjust the Hodge theory of the complex case to this non-elliptic context. For example, take $X_1$ and $X_2$ two differentiable manifolds having the same dimension: then, $X_1\times X_2$ has a natural \para-complex structure whose eigenbundles decomposition corresponds to the decomposition $T\left(X_1\times X_2\right)=TX_1\oplus TX_2$; it is straightforward to compute that the space $H^{0,0}_{\del_+}\left(X_1\times X_2\right)$ of $\del_+$-closed functions on $X_1\times X_2$ is not finite-dimensional, being
$$ H^{0,0}_{\del_+}\left(X_1\times X_2\right) \;\simeq\; \mathcal{C}^\infty\left(X_2\right) \;.$$

We recall that a \emph{\para-K\"ahler} structure on a \para-complex manifold $\left(X,\,K\right)$ is the datum of a \emph{$K$-compatible} symplectic  form (that is, a $K$-anti-invariant symplectic form): this is the \para-complex counterpart to K\"ahler notion in the complex case.

\subsection{\para-complex subgroups of cohomology}
Consider $\left(X,\,K\right)$ a compact almost \para-complex manifold and let $2n:=\dim X$.

The problem we are considering is when the decomposition
$$ \wedge^\bullet X \;=\; \bigoplus_{p,q} \formepm{p}{q}X \;=\; \wedge^{\bullet\,+}_KX \oplus \wedge^{\bullet\,-}_KX $$
moves to cohomology.\\
Here and later, we mime T.-J. Li and W. Zhang (see, e.g., \cite{li-zhang}). For any $p,q,\ell\in\N$, we define
$$ H^{(p,q)}_K\left(X;\R\right) \;:=\; \left\{\left[\alpha\right]\in H^{p+q}_{dR}\left(X;\R\right)\st \alpha\in\formepm{p}{q}X\right\} $$
and
\begin{eqnarray*}
H^{\ell\,+}_{K}\left(X;\R\right) &:=& \left\{\left[\alpha\right]\in H^\ell_{dR}\left(X;\R\right) \st K\alpha=\alpha\right\} \\
 & = & \left\{\left[\alpha\right]\in H^\ell_{dR}\left(X;\R\right) \st \alpha\in\wedge^{\ell\,+}_KX\right\} \;,\\[5pt]
H^{\ell\,-}_{K}\left(X;\R\right) &:=& \left\{\left[\alpha\right]\in H^\ell_{dR}\left(X;\R\right) \st K\alpha=-\alpha\right\} \\
& = & \left\{\left[\alpha\right]\in H^\ell_{dR}\left(X;\R\right) \st \alpha\in\wedge^{\ell\,-}_KX\right\} \;.
\end{eqnarray*}

\begin{rem}
 Note that, if $K$ is integrable, then, for any $\ell\in\N$,
$$ H^{\ell\,+}_{K} \;=\; \bigoplus_{p+q=\ell,\;q\text{ even}} H^{(p,q)}_K(X;\R) $$
and
$$ H^{\ell\,-}_{K} \;=\; \bigoplus_{p+q=\ell,\;q\text{ odd}} H^{(p,q)}_K(X;\R) \;. $$
\end{rem}

We introduce the following definitions. (We refer, e.g., to \cite{li-zhang, draghici-li-zhang-survey} and the references therein for precise definitions, motivations and results concerning the notion of \Cpf ness in almost complex geometry.)

\begin{defi}
 For $\ell\in\N$, an almost \para-complex structure $K$ on the manifold $X$ is said to be
\begin{itemize}
 \item \emph{\Cp\ at the $\ell$-th stage} if
  $$ H^{\ell\,+}_K\left(X;\R\right)\,\cap\,H^{\ell\,-}_K\left(X;\R\right) \;=\; \left\{0\right\} \;;$$
 \item \emph{\Cf\ at the $\ell$-th stage} if
  $$ H^{\ell\,+}_K\left(X;\R\right)\,+\,H^{\ell\,-}_K\left(X;\R\right) \;=\; H^\ell_{dR}\left(X;\R\right) \;;$$
 \item \emph{\Cpf\ at the $\ell$-th stage} if it is both \Cp\ at the $\ell$-th stage and \Cf\ at the $\ell$-th stage, or, in other words, if it satisfies the cohomological decomposition
  $$ H^\ell_{dR}\left(X;\R\right) \;=\; H^{\ell\,+}_K\left(X;\R\right)\,\oplus\,H^{\ell\,-}_K\left(X;\R\right) \;.$$
 \end{itemize}
\end{defi}

\subsection{\para-complex subgroups of homology}
Consider $\left(X,\,K\right)$ a compact almost \para-complex manifold and let $2n:=\dim X$. Denote by $D_\bullet X:=:D^{2n-\bullet}X$ the space of currents on $X$, that is, the topological dual space of $\wedge^\bullet X$. Define the de Rham homology $H_\bullet(X;\R)$ of $X$ as the homology of the complex $\left(D_\bullet X,\,\de\right)$, where $\de\colon D_\bullet X\to D_{\bullet-1}X$ is the dual operator of $\de\colon\wedge^{\bullet-1} X\to\wedge^\bullet X$. Note that there is a natural inclusion $T_\sspace\colon\wedge^\bullet X\hookrightarrow D^\bullet X:=:D_{2n-\bullet}X$ given by
$$ \eta\mapsto T_\eta:=\int_X \sspace \wedge \eta \;.$$
In particular, one has that $\de T_\eta=T_{\de\eta}$. Moreover, one can prove that $H^\bullet_{dR}(X;\R)\simeq H_{2n-\bullet}(X;\R)$ (see, e.g., \cite{derham}).\\
The action of $K$ on $\wedge^\bullet X$ induces, by duality, an action on $D_\bullet X$ (again denoted by $K$) and hence a decomposition
$$ D_\ell X \;=\; \bigoplus_{p+q=\ell} \correntipm{p}{q}X \;;$$
note that, for any $p,q\in\N$, the space $\correntipm{p}{q}X:=:\correntipmalto{n-p}{n-q}$ is the dual space of $\formepm{p}{q}X$ and that $T_\sspace\colon\formepm{p}{q}\hookrightarrow\correntipmalto{p}{q}X$. As in the smooth case, we set
$$ \correnti_{\bullet\,+}^{K}X \;:=\; \bigoplus_{q\text{ even}}\correntipm{\bullet}{q}X \qquad \text{ and } \qquad \correnti_{\bullet\,-}^{K}X \;:=\; \bigoplus_{q\text{ odd}}\correntipm{\bullet}{q}X \;,$$
so $K\lfloor_{\correnti_{\bullet\,\pm}^{K}X}=\pm\id$ for $\pm\in\{+,\,-\}$ and
$$ \correnti_\bullet X \;=\; \correnti_{\bullet\,+}^{K}X \,\oplus\, \correnti_{\bullet\,-}^{K}X \;.$$

For any $p,q,\ell\in\N$, we define
$$ H_{(p,q)}^K\left(X;\R\right) \;:=\; \left\{\left[\alpha\right]\in H_{p+q}\left(X;\R\right)\st \alpha\in\correntipm{p}{q}X\right\} $$
and
\begin{eqnarray*}
H_{\ell\,+}^{K}\left(X;\R\right) &:=& \left\{\left[\alpha\right]\in H_\ell\left(X;\R\right) \st K\alpha=\alpha\right\} \;, \\[5pt]
H_{\ell\,-}^{K}\left(X;\R\right) &:=& \left\{\left[\alpha\right]\in H_\ell\left(X;\R\right) \st K\alpha=-\alpha\right\} \;. \\[5pt]
\end{eqnarray*}

We introduce the following definitions.

\begin{defi}
 For $\ell\in\N$, an almost \para-complex structure $K$ on the manifold $X$ is said to be
\begin{itemize}
 \item \emph{\p\ at the $\ell$-th stage} if
  $$ H_{\ell\,+}^K\left(X;\R\right)\,\cap\,H_{\ell\,-}^K\left(X;\R\right) \;=\; \left\{0\right\} \;;$$
 \item \emph{\f\ at the $\ell$-th stage} if
  $$ H_{\ell\,+}^K\left(X;\R\right)\,+\,H_{\ell\,-}^K\left(X;\R\right) \;=\; H_\ell\left(X;\R\right) \;;$$
 \item \emph{\pf\ at the $\ell$-th stage} if it is both \p\ at the $\ell$-th stage and \f\ at the $\ell$-th stage, or, in other words, if it satisfies the homological decomposition
  $$ H_\ell\left(X;\R\right) \;=\; H_{\ell\,+}^K\left(X;\R\right)\,\oplus\,H_{\ell\,-}^K\left(X;\R\right) \;.$$
 \end{itemize}
\end{defi}

\subsection{Linkings between \para-complex decompositions in homology and cohomology}
We use the same argument as in \cite{li-zhang} to prove the following linkings between \Cpf\ and \pf\ concepts.

\begin{prop}[{see \cite[Proposition 2.30]{li-zhang} and also \cite[Theorem 2.1]{angella-tomassini}}]\label{prop:cpf-pf}
 Let $\left(X,\,K\right)$ be a $2n$-dimensional compact almost \para-complex manifold. Then, for every $\ell\in\N$, the following implications hold:
$$
\xymatrix{
\text{\Cf\ at the }\ell\text{-th stage}\ar@{=>}[r]\ar@{=>}[d] & \text{\p\ at the }\ell\text{-th stage}\ar@{=>}[d] \\
\text{\f\ at the }\left(2n-\ell\right)\text{-th stage}\ar@{=>}[r] & \text{\Cp\ at the }\left(2n-\ell\right)\text{-th stage}
}
$$
\end{prop}

\begin{proof}
 Since, for every $p,q\in\N$, we have $H^{(p,q)}_K(X;\R)\stackrel{T_\sspace}{\hookrightarrow}H_{(n-p,n-q)}^K(X;\R)$ and $H^\ell_{dR}(X;\R)\simeq H_{2n-\ell}(X;\R)$, we get that the two vertical arrows are obvious.\\
To prove the horizontal arrows, consider $\left\langle \sspace,\,\ssspace\right\rangle$ the duality paring $D_{\ell}X\times \wedge^{\ell}X\to\R$ or the induced non-degenerate pairing $H^\ell_{dR}(X;\R)\times H_\ell(X;\R)\to \R$. Suppose now that $K$ is \Cf\ at the $\ell$-th stage; if there exists $\mathfrak{c}=\left[\gamma_+\right]=\left[\gamma_-\right]\in H_{\ell\,+}^K(X;\R)\cap H_{\ell\,-}^K(X;\R)$ with $\gamma_+\in D_{\ell\,+}^KX$ and $\gamma_-\in D_{\ell\,-}^KX$, then
$$ \left\langle H^\ell(X;\R),\, \mathfrak{c} \right\rangle \;=\; \left\langle H^{\ell\,+}_K(X;\R),\, \left[\gamma_-\right]\right\rangle + \left\langle H^{\ell\,+}_K(X;\R),\, \left[\gamma_-\right]\right\rangle \;=\; 0$$
and therefore $\mathfrak{c}=0$ in $H_\ell(X;\R)$; hence $K$ is \p\ at the $\ell$-th stage.\\
A similar argument proves the bottom arrow.
\end{proof}

\begin{cor}\label{cor:cf-every-stage}
 If the almost \para-complex structure $K$ on the manifold $X$ is \Cf\ at every stage, then it is \Cpf\ at every stage and \pf\ at every stage.
\end{cor}

In particular, note that, on a compact $4$-dimensional manifold, being \Cf\ at the $2$-nd stage implies being \Cp\ at the $2$-nd stage.

\subsection{\para-complex decompositions in (co)homology for product manifolds}
Recall that, given $X_1$ and $X_2$ two differentiable compact manifolds with $\dim X_1=\dim X_2=:n$, the product $X_1\times X_2$ inherits a natural \para-complex structure $K$, given by the decomposition
$$ T\left(X_1\times X_2\right) \;=\; TX_1 \,\oplus\, TX_2 \;;$$
in other words, $K$ acts as $\id$ on $X_1$ and $-\id$ on $X_2$. For any $\ell\in\N$, using the K\"unneth formula, one gets
\begin{eqnarray*}
H^\ell_{dR}\left(X_1\times X_2;\R\right) &\simeq& \bigoplus_{p+q=\ell} H^p\left(X_1;\R\right)\,\otimes\, H^q\left(X_2;\R\right) \\[10pt]
&=& \underbrace{\left(\bigoplus_{p+q=\ell,\;q\text{ even}} H^p\left(X_1;\R\right)\,\otimes\, H^q\left(X_2;\R\right)\right)}_{\subseteq H^{\ell\,+}_K\left(X_1\times X_2;\,\R\right)}\\[5pt]
&&\oplus\,\underbrace{\left(\bigoplus_{p+q=\ell,\;q\text{ odd}} H^p\left(X_1;\R\right)\,\otimes\, H^q\left(X_2;\R\right)\right)}_{\subseteq H^{\ell\,-}_K\left(X_1\times X_2;\,\R\right)}\\[10pt]
&\subseteq& H^{\ell\,+}_K\left(X_1\times X_2;\R\right) \,+\, H^{\ell\,-}_K\left(X_1\times X_2;\R\right) \;.
\end{eqnarray*}

\noindent Therefore, using also Corollary \ref{cor:cf-every-stage}, one gets the following result (compare it with \cite[Proposition 2.6]{draghici-li-zhang-survey}).

\begin{thm}\label{thm:products}
 Let $X_1$ and $X_2$ be two equi-dimensional compact manifolds. Then the natural \para-complex structure on the product $X_1\times X_2$ is \Cpf\ at every stage and \pf\ at every stage.
\end{thm}

\section{Invariant \para-complex structures on solvmanifolds}\label{sec:invariant-cpf-para-complex-structures-on-solvmanifolds}
\subsection{Invariant \para-complex structures on solvmanifolds}
Let $X:=:\left.\Gamma\right\backslash G$ be a $2n$-dimensional \emph{solvmanifold} (respectively, \emph{nilmanifold}), that is, a compact quotient of a connected simply-connected solvable (respectively, nilpotent) Lie group $G$ by a co-compact discrete subgroup $\Gamma$. Set $\left(\mathfrak{g},\left[\sspace,\ssspace\right]\right)$ the Lie algebra that is naturally associated to the Lie group $G$; given a basis $\left\{e_1,\ldots,e_{2n}\right\}$ of $\mathfrak{g}$, the Lie algebra structure of $\mathfrak{g}$ is characterized by the \emph{structure constants} $\left\{c_{\ell m}^k\right\}_{\ell,m,k\in\{1,\ldots,2n\}}\subset\R$ such that, for any $k\in\left\{1,\ldots,2n\right\}$,
$$ \de_{\mathfrak{g}} e^k =: \sum_{\ell,m} c_{\ell m}^k \, e^\ell\wedge e^m \;, $$
where $\left\{e^1,\ldots,e^{2n}\right\}$ is the dual basis of $\duale{\mathfrak{g}}$ of $\left\{e_1,\ldots,e_{2n}\right\}$ and $\de_{\mathfrak{g}}:=:\de\colon\duale{\mathfrak{g}}\to \wedge^2\duale{\mathfrak{g}}$ is defined by
$$ \duale{\mathfrak{g}} \; \ni \; \alpha \mapsto \de_{\mathfrak{g}}\alpha(\sspace,\,\ssspace):=-\alpha\left(\left[\sspace,\,\ssspace\right]\right) \;\in\; \wedge^2\duale{\mathfrak{g}} \;.$$

Recall that a \emph{linear almost \para-complex} structure on $\mathfrak{g}$ is given by an endomorphism $K\in\End(\mathfrak{g})$ such that $K^2=\id_\mathfrak{g}$ and the eigenspaces $\mathfrak{g}^+$ and $\mathfrak{g}^-$ corresponding to the eigenvalues $1$ and $-1$ respectively of $K$ are equi-dimensional, i.e., $\dim_\R\mathfrak{g}^+=\dim_\R\mathfrak{g}^-=\frac{1}{2}\dim_\R\mathfrak{g}$.\\
Moreover, recall that a linear almost \para-complex structure on $\mathfrak{g}$ is said to be \emph{integrable} (and hence it is called a \emph{linear \para-complex} structure on $\mathfrak{g}$) if $\mathfrak{g}^+$ and $\mathfrak{g}^-$ are Lie-subalgebras of $\mathfrak{g}$, i.e.
$$ \left[\mathfrak{g}^+,\,\mathfrak{g}^+\right] \;\subseteq\; \mathfrak{g}^+ \qquad \text{ and }\qquad \left[\mathfrak{g}^-,\,\mathfrak{g}^-\right] \;\subseteq\; \mathfrak{g}^- \;.$$

Note that a $G$-invariant almost \para-complex structure on $X$ (that is, a \para-complex structure on $X$ induced by a \para-complex structure on $G$ which is invariant under the left-action of $G$ on itself given by translations) is determined by a linear almost \para-complex structure on $\mathfrak{g}$, equivalently, it is defined by the datum of two subspaces $\mathfrak{g}^+$ and $\mathfrak{g}^-$ of $\mathfrak{g}$ such that
$$ \mathfrak{g}\;=\; \mathfrak{g}^+\oplus\mathfrak{g}^-\quad \text{ and }\quad \dim_\R\mathfrak{g}^+=\dim_\R\mathfrak{g}^-=\frac{1}{2}\,\dim_\R\mathfrak{g} \;;$$
indeed, one can define $K\in\End(\mathfrak{g})$ as $K\lfloor_{\mathfrak{g}^+}=\id$ and $K\lfloor_{\mathfrak{g}^-}=-\id$ and then 
$K\in\End(TX)$ by translations. Note that the almost \para-complex structure $K$ on $X$ is integrable if and only if the linear almost \para-complex structure $K$ on $\mathfrak{g}$ is integrable.

\bigskip

\noindent {\itshape Notation.} To shorten the notation, we will refer to a given solvmanifold $X:=:\left.\Gamma\right\backslash G$ writing the structure equations of its Lie algebra: for example, writing
$$ X \;:=\; \left(0^4,\; 12+34\right)\;, $$
we mean that there exists a basis of the naturally associated Lie algebra $\mathfrak{g}$, let us say $\left\{e_1,\,\ldots,\,e_6\right\}$, whose dual will be denoted by $\left\{e^1,\,\ldots,\,e^6\right\}$ and with respect to which the structure equations are
$$
\left\{
\begin{array}{rcl}
 \de e^1 &=& \de e^2 \;=\; \de e^3 \;=\; \de e^4 \;=\; 0 \\[5pt]
 \de e^5 &=& e^{1}\wedge e^{2} \;=:\; e^{12} \\[5pt]
 \de e^6 &=& e^{1}\wedge e^{3} \;=:\; e^{13}
\end{array}
\right. \;,
$$
where we also shorten $e^{AB}:=e^A\wedge e^B$.
Recall that, by \cite{malcev}, given a nilpotent Lie algebra $\mathfrak{g}$ with {\itshape rational} structure constants, then the connected simply-connected Lie group $G$ naturally associated to $\mathfrak{g}$ admits a co-compact discrete subgroup $\Gamma$, and hence there exists a nilmanifold $X:=\left.\Gamma\right\backslash G$ whose Lie algebra is $\mathfrak{g}$.\\
With respect to the given basis $\left\{e_j\right\}_j$, writing that the (almost) \para-complex structure $K$ is defined as
$$ K \;:=\; \left(-\;+\;+\;-\;-\;+\right) $$
we mean that
$$ \mathfrak{g}^+ \;:=\; \R\left\langle e_2,\;e_3,\; e_6\right\rangle \qquad \text{ and }\qquad \mathfrak{g}^- \;:=\; \R\left\langle e_1,\;e_4,\; e_5\right\rangle \;.$$
Moreover, in writing the cohomology of $X$ (which is isomorphic to the cohomology of the complex $\left(\wedge^\bullet\duale{\mathfrak{g}},\,\de_\mathfrak{g}\right)$ if $X$ is completely-solvable, see \cite{hattori}), we list the harmonic representatives with respect to the invariant metric $g:=\sum_\ell e^\ell\odot e^\ell$ instead of their classes.\\
Dealing with \emph{invariant} objects on a $X$, we mean objects induced by objects on $G$ which are invariant under the left-action of $G$ on itself given by translations.

\subsection{The cohomology of completely-solvable solvmanifolds and its \para-complex subgroups}
Recall that the translation induces an isomorphism of differential algebras between the space of forms on $\duale{\mathfrak{g}}$ and the space $\wedge^\bullet_{\text{inv}}X$ of invariant differential forms on $X$:
$$ \left(\wedge^\bullet\duale{\mathfrak{g}},\,\de_{\mathfrak{g}}\right) \stackrel{\simeq}{\longrightarrow} \left(\wedge^\bullet_{\text{inv}}X,\,\de\lfloor_{\wedge^\bullet_{\text{inv}}X}\right) \;;$$
moreover, by Nomizu's and Hattori's theorems (see \cite{nomizu, hattori}), if $X$ is a nilmanifold or, more in general, a completely-solvable solvmanifold, then the natural inclusion
$$ \left(\wedge^\bullet_{\text{inv}}X,\,\de\lfloor_{\wedge^\bullet_{\text{inv}}X}\right) \hookrightarrow \left(\wedge^\bullet X,\,\de\right) $$
is a quasi-isomorphism, hence
$$  H^\bullet\left(\wedge^\bullet\duale{\mathfrak{g}},\,\de_{\mathfrak{g}}\right) \; \simeq \; H^\bullet\left(\wedge^\bullet_{\text{inv}}X,\,\de\lfloor_{\wedge^\bullet_{\text{inv}}X}\right) \;=:\; H^{\bullet}_{\text{inv}}(X;\R) \;\stackrel{\simeq}{\longrightarrow}\; H^\bullet_{dR}(X;\R) \;.$$
In this section, we study \para-complex decomposition in cohomology at the level of $H^\bullet\left({\mathfrak{g}};\R\right) := H^\bullet\left(\wedge^\bullet\duale{\mathfrak{g}},\,\de_{\mathfrak{g}}\right)$.

Recall that the linear almost \para-complex structure $K$ on $\mathfrak{g}$ defines a splitting $\mathfrak{g}=\mathfrak{g}^+\oplus \mathfrak{g}^-$ into eigenspaces and hence, for every $\ell\in\N$, one gets also the splitting
$$ \wedge^\ell\duale{\mathfrak{g}} \;=\; \bigoplus_{p+q=\ell}\wedge^{p}\duale{\left(\mathfrak{g}^+\right)}\otimes\wedge^{q}\duale{\left(\mathfrak{g}^-\right)} \;=:\; \bigoplus_{p+q=\ell}\formepm{p}{q}\duale{\mathfrak{g}} \;,$$
where, for any $p,q\in\N$, one has $K\lfloor_{\formepm{p}{q}\duale{\mathfrak{g}}}=(+1)^p\,(-1)^q\,\id$; we introduce also the splitting of the differential forms into their $K$-invariant and $K$-anti-invariant components:
$$ \wedge^\bullet\duale{\mathfrak{g}} \;=\; \wedge^{\bullet\,+}_K\duale{\mathfrak{g}}\,\oplus\,\wedge^{\bullet\,-}_K\duale{\mathfrak{g}} $$
where
$$ \wedge^{\bullet\,+}_K\duale{\mathfrak{g}} \;:=\; \bigoplus_{q\text{ even}}\formepm{\bullet}{q}\duale{\mathfrak{g}} \quad \text{ and }\quad \wedge^{\bullet\,-}_K\duale{\mathfrak{g}} \;:=\; \bigoplus_{q\text{ odd}}\formepm{\bullet}{q}\duale{\mathfrak{g}} \;. $$

As already done for manifolds, for any $p,q,\ell\in\N$, we define
$$ H^{(p,q)}_K\left({\mathfrak{g}};\R\right) \;:=\; \left\{\left[\alpha\right]\in H^{p+q}\left(\mathfrak{g};\R\right)\st \alpha\in\formepm{p}{q}\duale{\mathfrak{g}}\right\} $$
and
\begin{eqnarray*}
H^{\ell\,+}_{K}\left(\mathfrak{g};\R\right) &:=& \left\{\left[\alpha\right]\in H^\ell\left({\mathfrak{g}};\R\right) \st K\alpha=\alpha\right\}
\;, \\[5pt]
H^{\ell\,-}_{K}\left(\mathfrak{g};\R\right) &:=& \left\{\left[\alpha\right]\in H^\ell\left({\mathfrak{g}};\R\right) \st K\alpha=-\alpha\right\}
\;.
\end{eqnarray*}

We give the following definition.
\begin{defi}
 For $\ell\in\N$, a linear almost \para-complex structure on the Lie algebra $\mathfrak{g}$ is said to be
\begin{itemize}
 \item \emph{linear \Cp\ at the $\ell$-th stage} if
  $$ H^{\ell\,+}_K\left(\mathfrak{g};\R\right)\,\cap\,H^{\ell\,-}_K\left(\mathfrak{g};\R\right) \;=\; \left\{0\right\} \;;$$
 \item \emph{linear \Cf\ at the $\ell$-th stage} if
  $$ H^{\ell\,+}_K\left(\mathfrak{g};\R\right)\,+\,H^{\ell\,-}_K\left(\mathfrak{g};\R\right) \;=\; H^\ell\left(\mathfrak{g};\R\right) \;;$$
 \item \emph{linear \Cpf\ at the $\ell$-th stage} if it is both \Cp\ at the $\ell$-th stage and \Cf\ at the $\ell$-th stage, or, in other words, if it satisfies the cohomological decomposition
  $$ H^\ell\left(\mathfrak{g};\R\right) \;=\; H^{\ell\,+}_K\left(\mathfrak{g};\R\right)\,\oplus\,H^{\ell\,-}_K\left(\mathfrak{g};\R\right) \;.$$
 \end{itemize}
\end{defi}

\medskip

Given a completely-solvable solvmanifold, we want now to make clear the connection between the \Cpf ness of an invariant almost \para-complex structure and the linear \Cpf ness of the corresponding linear almost \para-complex structure on the associated Lie algebra.

\noindent We need the following result by J. Milnor.

\begin{lemma}[{\cite[Lemma 6.2]{milnor}}]\label{lemma:milnor}
 Any connected Lie group that admits a discrete subgroup with compact quotient is unimodular and in particular admits a bi-invariant volume form $\eta$.
\end{lemma}

\noindent The previous Lemma is used to prove the following result, for which we refer to \cite{fino-grantcharov} by A. Fino and G. Grantcharov.

\begin{lemma}[{see \cite[Theorem 2.1]{fino-grantcharov}}]\label{lemma:fino-grantcharov}
Let $X:=:\left.\Gamma\right\backslash G$ be a solvmanifold and call $\mathfrak{g}$ the Lie algebra that is naturally associated to the connected simply-connected Lie group $G$. Denote by $K$ an invariant almost \para-complex structure on $X$ or equivalently the associated linear almost \para-complex structure on $\mathfrak{g}$.
Let $\eta$ be the bi-invariant volume form on $G$ given by Lemma \ref{lemma:milnor} and suppose that $\int_X\eta=1$. Define the map
$$ \mu\colon \wedge^\bullet X \to \wedge^\bullet_{\text{inv}}X\;,\qquad \mu(\alpha)\;:=\;\int_X \alpha\lfloor_m \, \eta(m) \;.$$
One has that
$$ \mu\lfloor_{\wedge^\bullet_{\text{inv}}X}\;=\;\id\lfloor_{\wedge^\bullet_{\text{inv}}X} $$
and that
$$ \de\left(\mu(\sspace)\right) \;=\; \mu\left(\de\sspace\right) \qquad \text{ and } \qquad K\left(\mu(\sspace)\right) \;=\; \mu\left(K\sspace\right) \;.$$
\end{lemma}

Then we can prove the following result. Note that, with slight and obvious modifications, it holds also for almost complex structures: a similar result for almost complex structures has been obtained also by A. Tomassini and A. Fino in \cite[Theorem 3.4]{fino-tomassini}.

\begin{prop}\label{prop:linear-cpf-invariant-cpf}
Let $X:=:\left.\Gamma\right\backslash G$ be a completely-solvable solvmanifold and call $\mathfrak{g}$ the Lie algebra that is naturally associated to the connected simply-connected Lie group $G$. Denote by $K$ an invariant almost \para-complex structure on $X$ or equivalently the associated linear almost \para-complex structure on $\mathfrak{g}$.
Then, for every $\ell\in\N$ and for $\pm\in\{+,\,-\}$, the injective map
$$ H^{\ell\,\pm}_K(\mathfrak{g};\R) \to H^{\ell\,\pm}_K(X;\R) $$
induced by translations is an isomorphism.\\
Furthermore, for every $\ell\in\N$, the linear \para-complex structure $K\in\End(\mathfrak{g})$ is linear \Cp\ (respectively, linear \Cf) at the $\ell$-th stage if and only if the \para-complex structure $K\in\End(TX)$ is \Cp\ (respectively, \Cf) at the $\ell$-th stage.
\end{prop}

\begin{proof}
Consider the map $\mu\colon\wedge^\bullet X \to \wedge^\bullet_{\text{inv}}X$ defined in Lemma \ref{lemma:fino-grantcharov}. The thesis follows from the following three observations.\\
Since $\de\left(\mu(\sspace)\right)=\mu\left(\de\sspace\right)$, one has that $\mu$ sends $\de$-closed (respectively, $\de$-exact) forms to $\de$-closed (respectively, $\de$-exact) invariant forms and so it induces a map
$$ \mu\colon H^\bullet_{dR}(X;\R)\to H^\bullet_{\text{inv}}(X;\R)\;\simeq\; H^\bullet\left(\mathfrak{g};\R\right) \;.$$
Since $K\left(\mu(\sspace)\right)=\mu\left(K\sspace\right)$, for $\pm\in\{+,\,-\}$, one has
$$ \mu\left(\wedge^{\bullet\,\pm}_KX\right) \;\subseteq\; \wedge^{\bullet\,\pm}_{K\,\text{inv}}X \;,$$
where $\wedge^{\bullet\,\pm}_{K\,\text{inv}}X:=\wedge^{\bullet\,\pm}_{K}X\cap\wedge^\bullet_{\text{inv}}X\simeq \wedge^{\bullet\,\pm}_K\duale{\mathfrak{g}}$, hence
$$ \mu\left(H^{\bullet\,\pm}_K(X;\R)\right) \;\subseteq\; H^{\bullet\,\pm}_K\left(\mathfrak{g};\R\right) \;. $$
Lastly, since $X$ is a completely-solvable solvmanifold, its cohomology is isomorphic to the invariant one (see \cite{hattori}) and hence the condition $\mu\lfloor_{\wedge^\bullet_{\text{inv}}X}=\id\lfloor_{\wedge^\bullet_{\text{inv}}X}$ gives that $\mu$ is the identity in cohomology.
\end{proof}

\section{\Cpf ness of invariant \para-complex structures on solvmanifolds}\label{sec:cpfness-solvmanifolds}
\subsection{Some examples of non-\Cpf\ (almost) \para-complex nilmanifolds}
In this section, we use the notation and the results in Section \ref{sec:invariant-cpf-para-complex-structures-on-solvmanifolds} to provide examples of invariant (almost) \para-complex structures on nilmanifolds.

Firstly, we give two examples of non-\Cp\ or non-\Cf\ nilmanifolds admitting \para-K\"ahler structures.

\begin{ex}\label{es 2.5}{\itshape There exists a $6$-dimensional \para-complex nilmanifold that is \Cp\ at the $2$-nd stage and non-\Cf\ at the $2$-nd stage and admits a \para-K\"ahler structure.}\\
Indeed, take the nilmanifold
 $$ X \;:=\; \left(0^4,\, 12,\, 13\right) $$
and define the invariant \para-complex structure $K$ setting
$$ K \;:=\; \left(-\,+\,+\,-\,-\,+\right) \;.$$
By Nomizu's theorem (see \cite{nomizu}), the de Rham cohomology of $X$ is given by
$$
H^2_{dR}(X;\R) \;\simeq\; H^2_{dR}(\mathfrak{g};\R) \;=\; \R\left\langle e^{14},\; e^{15},\; e^{16},\; e^{23},\; e^{24},\; e^{25},\; e^{34},\; e^{36},\; e^{26}+e^{35} \right\rangle \;.
$$
Note that
$$ H^{2\,+}_K\left(\mathfrak{g};\R\right) \;=\; \R\left\langle e^{14},\;e^{15},\;e^{23},\;e^{36} \right\rangle $$
and
$$ H^{2\,-}_K\left(\mathfrak{g};\R\right) \;=\; \R\left\langle e^{16},\;e^{24},\;e^{25},\;e^{34} \right\rangle \;, $$
since no invariant representative in the class $\left[e^{26}+e^{35}\right]$ is of pure type with respect to $K$ (indeed, the space of invariant $\de$-exact $2$-forms is $\R\left\langle e^{12},\, e^{13}\right\rangle$). It follows that $K\in\End\left(\mathfrak{g}\right)$ is linear \Cp\ at the $2$-nd stage and linear non-\Cf\ at the $2$-nd stage and hence, by Proposition \ref{prop:linear-cpf-invariant-cpf}, $K\in\End(TX)$ is \Cp\ at the $2$-nd stage (being $K$ Abelian, see Definition \ref{def:abelian}, one can also argue using Corollary \ref{cor:abelian}) and non-\Cf\ at the $2$-nd stage.\\
Moreover, we observe that
$$ \omega \;:=\; e^{16}+e^{25}+e^{34} $$
is a symplectic form compatible with $K$, hence $\left(X,\,K,\,\omega\right)$ is a \para-\kal\ manifold.
\end{ex}

\begin{ex}\label{es 2.6} {\itshape There exists a $6$-dimensional \para-complex nilmanifold that is non-\Cp\ at the $2$-nd stage (and hence non-\Cf\ at the $4$-th stage) and admitting a \para-K\"ahler structure.}\\
Indeed, take the nilmanifold $X$ defined by
 $$ X \;:=\; \left(0^3,\;12,\;13+14,\;24\right) $$
and define the invariant \para-complex structure $K$ as
 $$ K \;:=\; \left(+\;-\;+\;-\;+\;-\right) \;.$$
(Note that, since $\left[e_2,e_4\right]=-e_6$, one has that $\left[\mathfrak{g}^-,\,\mathfrak{g}^-\right]\neq\left\{0\right\}$ and hence $K$ is not Abelian, see Definition \ref{def:abelian}.)\\
We have
$$ H^{2\,+}_K\left(\mathfrak{g};\R\right) \;\ni\; \left[e^{13}\right] \;=\; \left[e^{13}-\de e^{5}\right] \;=\; -\left[e^{14}\right] \;\in\; H^{2\,-}_K\left(\mathfrak{g};\R\right) $$
and therefore we get that
$$ 0 \;\neq\; \left[e^{13}\right] \;\in\; H^{2\,+}_K\left(\mathfrak{g};\R\right)\,\cap\,H^{2\,-}_K\left(\mathfrak{g};\R\right) \;,$$
namely, $K\in\End(\mathfrak{g})$ is not linear \Cp\ at the $2$-nd stage, hence $K\in\End(TX)$ is not \Cp\ at the $2$-nd stage; furthermore, by Proposition \ref{prop:cpf-pf}, we have also that $K$ is not \Cf\ at the $4$-th stage.\\
Moreover, we observe that
$$ \omega \;:=\; e^{16}+e^{25}+e^{34} $$
is a symplectic form compatible with $K$, hence $\left(X,\,K,\,\omega\right)$ is a \para-\kal\ manifold.
\end{ex}

The previous two examples prove the following result, in contrast with the complex case (see, e.g., \cite{li-zhang}). (Note that higher-dimensional examples of \para-K\"ahler non-\Cf, respectively non-\Cp, at the $2$-nd stage structures can be obtained taking products with standard \para-complex tori.)

\begin{prop}\label{prop:para-kahler}
 Admitting a \para-K\"ahler structure does not imply neither being \Cp\ at the $2$-nd stage nor being \Cf\ at the $2$-nd stage.
\end{prop}

\medskip

The following example shows that \cite[Theorem 2.3]{draghici-li-zhang} by T. Dr\v{a}ghici, T.-J. Li and W. Zhang (saying that every almost complex structure on a $4$-dimensional compact manifold induces an almost complex decomposition at the level of the real second de Rham cohomology group) does not hold, in general, in the almost \para-complex case.

\begin{ex}\label{es 2.8}
{\itshape There exists a $4$-dimensional almost \para-complex nilmanifold which is non-\Cpf\ at the $2$-nd stage.}\\
Indeed, take the nilmanifold $X$ defined by
 $$ X \;:=\; \left(0,0,12,13\right) $$
(namely, the product of the Heisenberg group and $\R$) and define the invariant almost \para-complex structure by the eigenspaces
\begin{equation*}
\mathfrak{g}^+\;:=\;\R\left\langle e_1,\, e_4-e_2\right\rangle \quad \text{ and }\quad \mathfrak{g}^-\;:=\;\R\left\langle e_2,\, e_3\right\rangle.
\end{equation*}
Note that $K$ is not integrable, since $\left[e_1,e_4-e_2\right]=e_3$.\\
Note that we have
$$ H^{2\,+}_K\left(\mathfrak{g};\R\right)\ni\left[e^{14}\right]=\left[e^{14}+\de e^{3}\right]=\left[e^{14}+e^{12}\right]= \left[e^1\wedge(e^4+e^2)\right]\in H^{2\,-}_K\left(\mathfrak{g};\R\right) $$
and therefore we get that
$$ 0 \;\neq\; \left[e^{14}\right] \;\in\; H^{2\,+}_K\left(\mathfrak{g};\R\right)\,\cap\,H^{2\,-}_K\left(\mathfrak{g};\R\right) \;,$$
Then, $K$ is not \Cp\ at the $2$-nd stage and, by Proposition \ref{prop:cpf-pf}, is not \Cf\ at the $2$-nd stage.
\end{ex}

\subsection{\Cpf ness of low-dimensional \para-complex solvmanifolds}\label{subsec:Cpf-low-dim}
Let $\left(\mathfrak{a},\,\left[\sspace,\ssspace\right]\right)$ be a Lie algebra and consider the \emph{lower central series} $\left\{\mathfrak{a}_n\right\}_{n\in\N}$ defined, by induction on $n\in\N$, as
$$
\left\{
\begin{array}{rcll}
 \mathfrak{a}^0 &:=& \mathfrak{a} & \\[5pt]
 \mathfrak{a}^{n+1} &:=& \left[\mathfrak{a}^n,\,\mathfrak{a}\right] & \text{ for }n\in\N
\end{array}
\right. \;;
$$
note that $\left\{\mathfrak{a}_n\right\}_{n\in\N}$ is a descending sequence of Lie algebras:
$$ \mathfrak{a}\;=\; \mathfrak{a}^0 \;\supseteq \mathfrak{a}^1 \;\supseteq\; \cdots \;\supseteq\; \mathfrak{a}^{j-1} \;\supseteq\; \mathfrak{a}^j \;\supseteq\; \cdots \;.$$
Recall that the \emph{nilpotent step} of $\mathfrak{a}$ is defined as
$$ s\left(\mathfrak{a}\right) \;:=\; \inf\left\{n\in\N \st \mathfrak{a}^n=0\right\} \;, $$
so $s\left(\mathfrak{a}\right)<+\infty$ means that $\mathfrak{a}$ is nilpotent.

In particular, if the linear \para-complex structure $K$ on the Lie algebra $\mathfrak{g}$ induces the decomposition $\mathfrak{g}=\mathfrak{g}^+\oplus\mathfrak{g}^-$, we consider
$$ s^+ \;:=\; s\left(\mathfrak{g}^+\right) \qquad \text{ and }\qquad s^- \;:=\; s\left(\mathfrak{g}^-\right) \;;$$
since $\mathfrak{g}^+\subset\mathfrak{g}$ and $\mathfrak{g}^-\subset\mathfrak{g}$, we have obviously that
$$ s^+\;\leq\; s\left(\mathfrak{g}\right) \qquad \text{ and } \qquad s^- \;\leq\; s\left(\mathfrak{g}\right)  \;. $$

We start with the following easy lemma.
\begin{lemma}\label{lemma:s+-}
 Let $\mathfrak{g}$ be a $2n$-dimensional nilpotent Lie algebra, that is, $s\left(\mathfrak{g}\right)<+\infty$. Let $K$ be a linear \para-complex structure on $\mathfrak{g}$, inducing the decomposition $\mathfrak{g}=\mathfrak{g}^+\oplus\mathfrak{g}^-$. Then, setting $s^\pm:=s\left(\mathfrak{g}^{\pm}\right)$ for $\pm\in\left\{+,\,-\right\}$, we have
$$ 1 \;\leq\; s^+ \;\leq\; n-1 \qquad \text{ and } \qquad 1 \;\leq\; s^- \;\leq\; n-1 \;.$$
\end{lemma}

\begin{proof}
 The proof follows easily observing that, for $\pm\in\left\{+,\;-\right\}$, we have
$$
\left\{
\begin{array}{rcll}
 \dim_\R \left(\mathfrak{g}^\pm\right)^0 &=& n & \\[5pt]
 \dim_\R \left(\mathfrak{g}^\pm\right)^k &\leq& \max\left\{n-k-1,\,0\right\} & \text{ for } k\geq 1\\[5pt]
\end{array}
\right. \;,
$$
as a consequence of the nilpotent condition and of the integrability property.
\end{proof}

We start with the following result, to be compared with Theorem \ref{thm:products}.

\begin{prop}\label{prop:g+g-comm}
 Let $\mathfrak{g}$ be a Lie algebra. If $K$ is a linear \para-complex structure on $\mathfrak{g}$ with eigenspaces $\mathfrak{g}^+$ and $\mathfrak{g}^-$ such that $\left[\mathfrak{g}^+,\,\mathfrak{g}^-\right]=\{0\}$, then $K$ is linear \Cpf\ at every stage.
\end{prop}

\begin{proof}
 Since $\left[\mathfrak{g}^+,\,\mathfrak{g}^-\right]=\{0\}$, one can write $\mathfrak{g}=\mathfrak{g}^+ \times \mathfrak{g}^-$ and, using K\"unneth formula as in Theorem \ref{thm:products}, one gets the thesis.
\end{proof}

Therefore, from Proposition \ref{prop:linear-cpf-invariant-cpf}, one gets the following corollary.
\begin{cor}
 Let $X:=:\left.\Gamma\right\backslash G$ be a completely-solvable solvmanifold endowed with an invariant \para-complex structure $K$. Call $\mathfrak{g}$ the Lie algebra naturally associated to the Lie group $G$ and consider the linear \para-complex structure $K\in\End(\mathfrak{g})$ induced by $K\in\End(TX)$. Suppose that the eigenspaces $\mathfrak{g}^+$ and $\mathfrak{g}^-$ of $K\in\End(\mathfrak{g})$ satisfy $\left[\mathfrak{g}^+,\,\mathfrak{g}^-\right]=\{0\}$. Then $K$ is \Cpf\ at every stage and \pf\ at every stage.
\end{cor}

Recall the following definition.

\begin{defi}\label{def:abelian}
 A linear \para-complex structure on a Lie algebra $\mathfrak{g}$ is said to be \emph{Abelian} if the induced decomposition $\mathfrak{g}=\mathfrak{g}^+\oplus\mathfrak{g}^-$ satisfies $\left[\mathfrak{g}^+,\,\mathfrak{g}^+\right]=\{0\}=\left[\mathfrak{g}^-,\,\mathfrak{g}^-\right]$, namely, $s^+=1=s^-$.
\end{defi}

\begin{rem}\label{oss 2.11}
Note that every linear \para-complex structure on a $4$-dimensional nilpotent Lie algebra is Abelian, as a consequence of Lemma \ref{lemma:s+-}.
\end{rem}

\begin{thm}\label{thm 2.12}
 Let $\mathfrak{g}$ be a Lie algebra and $K$ be a linear Abelian \para-complex structure on $\mathfrak{g}$. Then $K$ is linear \Cp\ at the $2$-nd stage.
\end{thm}

\begin{proof}
 Denote by $\pi^+\colon\wedge^\bullet\duale{\mathfrak{g}}\to\wedge^{\bullet\,+}_K\duale{\mathfrak{g}}$ the map that gives the $K$-invariant component of a given form.
 Recall that $\de\eta:=-\eta\left(\left[\sspace,\,\ssspace\right]\right)$ for every $\eta\in\wedge^1\duale{\mathfrak{g}}$; therefore, since $\left[\mathfrak{g}^+,\,\mathfrak{g}^+\right]=0$ and $\left[\mathfrak{g}^-,\,\mathfrak{g}^-\right]=0$, we have that $$\pi^+_K\left(\imm\left(\de\colon\wedge^1\duale{\mathfrak{g}}\to\wedge^2\duale{\mathfrak{g}}\right)\right)\;=\;\{0\} \;.$$
 Suppose that there exists $\left[\gamma^+\right]=\left[\gamma^-\right]\in H^{2\,+}_K\left(\mathfrak{g};\,\R\right)\cap H^{2\,-}_K\left(\mathfrak{g};\,\R\right)$, where $\gamma^+\in\wedge^{2\,+}_K\duale{\mathfrak{g}}$ and $\gamma^-\in\wedge^{2\,-}_K\duale{\mathfrak{g}}$; let $\alpha\in\wedge^1\duale{\mathfrak{g}}$ be such that $\gamma^+=\gamma^-+\de\alpha$. Since $\pi^+_K\left(\de\alpha\right)=0$, we have that $\gamma^+=0$ and hence $\left[\gamma^+\right]=0$, so $K$ is linear \Cp\ at the $2$-nd stage.
\end{proof}

\begin{rem}\label{rem:partial-abelian}
 We note that the condition of $K$ being Abelian in Theorem \ref{thm 2.12} can not be dropped, not even partially. In fact, Example \ref{es 2.17} shows that the Abelian assumption just on $\mathfrak{g}^-$ is not sufficient to have \Cp ness at the $2$-nd stage. Another example on a (non-unimodular) solvable Lie algebra is given below.
\end{rem}

\begin{ex}
{\itshape There exists a $4$-dimensional (non-unimodular) solvable Lie algebra with a non-Abelian \para-complex structure that is not linear \Cp\ at the $2$-nd stage.}\\
Consider the $4$-dimensional solvable Lie algebra defined by
$$ \mathfrak{g} \;:=\; \left(0,\;0,\;0,\; 13+34\right) \;;$$
note that $\mathfrak{g}$ is not unimodular, since $\de e^{124}=e^{1234}$, see Lemma \ref{lemma:unimod}.\\
Set the linear \para-complex structure
$$ K\;:=\; \left( +\;+\;-\;- \right) \;;$$
note that $K$ is not Abelian, since $\left[\mathfrak{g}^+,\,\mathfrak{g}^+\right]=0$ but $\left[\mathfrak{g}^-,\,\mathfrak{g}^-\right]=\R\left\langle e_3\right\rangle\neq \{0\}$.\\
A straightforward computation yields that $\mathfrak{g}$ is linear \Cf\ (in fact, $H^2\left(\mathfrak{g};\R\right)=\R\left\langle e^{12},\,e^{34}\right\rangle\oplus\left\langle e^{23}\right\rangle$ and $H^+_K\left(\mathfrak{g};\R\right)=\R\left\langle e^{12},\,e^{34}\right\rangle$, $H^-_K\left(\mathfrak{g};\R\right)=\R\left\langle e^{23} \right\rangle$) but linear non-\Cp, since
$$ H^{2\,+}_K\left(\mathfrak{g};\R\right) \;\ni\; \left[e^{34}\right] \;=\; \left[e^{34}-\de e^{4}\right] \;=\; -\left[e^{13}\right] \;\in\; H^{2\,-}_K\left(\mathfrak{g};\R\right) $$
and $\left[e^{34}\right]\neq 0$.
\end{ex}

Saying that a \para-complex structure on a solvmanifold is \emph{Abelian}, we will mean that the associated linear \para-complex structure on the corresponding Lie algebra is Abelian.\\
As a corollary of Theorem \ref{thm 2.12} and using Proposition \ref{prop:linear-cpf-invariant-cpf}, we get the following result.

\begin{cor}\label{cor:abelian}
 Let $X:=:\left.\Gamma\right\backslash G$ be a completely-solvable solvmanifold endowed with an invariant Abelian \para-complex structure $K$. Then $K$ is \Cp\ at the $2$-nd stage.
\end{cor}

\begin{rem}
 For a \para-complex structure on a compact manifold, being Abelian or being \Cp\ at the $2$-nd stage is not a sufficient condition to have \Cf ness at the $2$-nd stage. Indeed, Example \ref{es 2.5} provides a \para-complex structure on a $6$-dimensional solvmanifold that is Abelian, \Cp\ at the $2$-nd stage and non-\Cf\ at the $2$-nd stage.
\end{rem}

In particular, recalling Remark \ref{oss 2.11}, invariant \para-complex structures on $4$-di\-men\-sio\-nal nilmanifolds are \Cp\ at the $2$-nd stage. While for invariant Abelian \para-complex structures on higher-dimensional nilmanifolds we can not hope to have, in general, \Cf ness at the $2$-nd stage (see Example \ref{es 2.5}), for $4$-dimensional nilmanifolds we can prove that every invariant \para-complex structure is in fact also \Cf\ at the $2$-nd stage, see Theorem \ref{thm:4-dim}: to prove this fact, we need the following lemmata. The first one is a classical result.

\begin{lemma}(see, e.g., \cite{greub-halperin-vanstone})\label{lemma:unimod}
 Let $\mathfrak{g}$ be a unimodular Lie algebra of dimension $n$. Then
$$ \de\lfloor_{\wedge^{n-1}\duale{\mathfrak{g}}} \;=\; 0 \;.$$
\end{lemma}

\begin{lemma}\label{lemma:n0-puro}
 Let $\mathfrak{g}$ be a unimodular Lie algebra of dimension $2n$ endowed with an Abelian linear \para-complex structure $K$. Then
$$ \de\lfloor_{\formepm{n}{0}\duale{\mathfrak{g}}\,\oplus\,\formepm{0}{n}\duale{\mathfrak{g}}} \;=\; 0 \;.$$
\end{lemma}

\begin{proof}
 Consider the bases
$$ \duale{\left(\mathfrak{g}^+\right)} \;=\; \R\left\langle e^1,\, \ldots,\, e^n\right\rangle \qquad \text{ and }\qquad \duale{\left(\mathfrak{g}^-\right)} \;=\; \R\left\langle f^1,\, \ldots,\, f^n\right\rangle $$
where $\mathfrak{g}=\mathfrak{g}^+\oplus\mathfrak{g}^-$ is the decomposition into eigenspaces induced by $K$. Being $K$ Abelian, the general structure equations are of the form
$$
\left\{
\begin{array}{rcl}
 \de e^j &=:& \sum_{h,\,k=1}^n a^j_{hk}\, e^h\wedge f^k \\[5pt]
 \de f^j &=:& \sum_{h,\,k=1}^n b^j_{hk}\, e^h\wedge f^k
\end{array}
\right.
$$
varying $j\in\{1,\ldots,n\}$ and where $\left\{a^j_{hk},\,b^j_{hk}\right\}_{j,h,k}\subset\R$.\\
A straightforward computation yields
$$ \de\left(e^1\wedge\cdots\wedge e^n\right) \;=\; \left(-1\right)^n \sum_{k=1}^n \left(\sum_{\ell=1}^n a^\ell_{\ell k}\right) e^1\wedge\cdots\wedge e^n\wedge f^k $$
where, for any $k\in\{1,\ldots,n\}$,
$$ \sum_{\ell=1}^n a^\ell_{\ell k} \;=\; 0 \;,$$
since it is the coefficient of
$$ \de\left( e^1\wedge\cdots\wedge e^n\wedge f^1\wedge\cdots\wedge f^{k-1}\wedge f^{k+1}\wedge\cdots\wedge f^n\right) \;=\; 0 \;, $$
by Lemma \ref{lemma:unimod}.
\end{proof}

We can now prove the following result.

\begin{thm}\label{thm:4-dim}
 Every invariant \para-complex structure on a $4$-dimensional nilmanifold is \Cpf\ at the $2$-nd stage and hence also \pf\ at the $2$-nd stage. 
\end{thm}

\begin{proof}
 The \Cp ness at the $2$-nd stage follows from Remark \ref{oss 2.11} and Corollary \ref{cor:abelian}.
 From Lemma \ref{lemma:n0-puro} one gets that, on every $4$-dimensional \para-complex nilmanifold, the \para-complex invariant component of an invariant $2$-form is closed and hence also the \para-complex anti-invariant component of a closed invariant $2$-form is closed, hence the linear \para-complex structure is linear \Cf\ at the $2$-nd stage; by Proposition \ref{prop:linear-cpf-invariant-cpf}, the \para-complex structure is hence \Cf\ at the $2$-nd stage. Lastly, the \pf ness at the $2$-nd stage follows from Proposition \ref{prop:cpf-pf}.
\end{proof}

\begin{rem}
 We note that Theorem \ref{thm:4-dim} is optimal. Indeed, we can not grow dimension (see Example \ref{es 2.5} and Example \ref{es 2.6}), nor change the nilpotent hypothesis with solvable condition (see Example \ref{es 2.17}), nor drop the integrability condition on the \para-complex structure (see Example \ref{es 2.8}).
\end{rem}

\section{Small deformations of \para-complex structures}\label{sec:deformations}
In this section, we study explicit examples of deformations of \para-complex structures on nilmanifolds and solvmanifolds. We refer to \cite{medori-tomassini, rossi} for more results about deformations of \para-complex structures.

The following example shows a curve $\left\{K_t\right\}_{t\in\R}$ of \para-complex structure on a $4$-dimensional solvmanifold; while $K_0$ is linear \Cpf\ at the $2$-nd stage and admits a \para-K\"ahler structure, for $t\neq 0$ one proves that $K_t$ is neither \Cp\ at the $2$-nd stage nor \Cf\ at the $2$-nd stage and it does not admit a \para-K\"ahler structure. In particular, this curve provides an example of the instability of \para-K\"ahlerness under small deformations of the \para-complex structure and it proves also that the nilpotency condition in Theorem \ref{thm:4-dim} can not be dropped out.

\begin{ex}\label{es 2.17}{\itshape There exists a $4$-dimensional solvmanifold endowed with an invariant \para-complex structure such that it is \Cpf\ at the $2$-nd stage, it admits a \para-K\"ahler structure and it has small deformations that are neither \para-K\"ahler nor \Cpf\ at the $2$-nd stage.}\\
Consider the $4$-dimensional solvmanifold defined by
$$ X \;:=\; \left(0,\,0,\,23,\,-24 \right) $$
(see, e.g., \cite{bock}).\\
By Hattori's theorem (see \cite{hattori}), it is straightforward to compute
$$ H^2_{dR}(X;\R) \;=\; \R\left\langle e^{12},\,e^{34}\right\rangle \;.$$
For every $t\in\R$, consider the invariant \para-complex structure
$$
K_t \;:=\;
\left(
\begin{array}{cccc}
 -1 & 0 & 0 & 0\\
 0 & 1 & 0 & -2t \\
 0 & 0 & 1 & 0 \\
 0 & 0 & 0 & -1
\end{array}
\right) \;.
$$
In particular, for $t=0$, we have
$$ K_0 \;=\; \left( - \, + \, + \, -\right) \;. $$
It is straightforward to check that $K_0$ is \Cpf\ at the $2$-nd stage (note however that $K_0$ is not Abelian): in fact,
\begin{equation*}
 H^{2\,+}_{K_0}(X;\R) \;=\; \left\{0\right\} \qquad \text{ and } \qquad  H^{2\,-}_{K_0}(X;\R) \;=\; H^2_{dR}(X;\R) \;;
\end{equation*}
in particular, we have
$$ \dim_\R H^{2\,+}_{K_0}(X;\R) \;=\; 0\;, \qquad \dim_\R H^{2\,-}_{K_0}(X;\R) \;=\; 2.$$
For every $t\in\R$, we have that
\begin{equation*}
 \mathfrak{g}^+_{K_t}\;=\;\R\left\langle e_2, \, e_3 \right\rangle \quad \text{ and }\quad \mathfrak{g}^-_{K_t}\;=\;\R\left\langle e_1,\,e_4+t\,e_2\right\rangle \;:
\end{equation*}
in particular, $\left[\mathfrak{g}^+_{K_t},\,\mathfrak{g}^+_{K_t}\right]=\R\left\langle e_3\right\rangle\subseteq \mathfrak{g}^+_{K_t}$ and $\left[\mathfrak{g}^-_{K_t},\,\mathfrak{g}^-_{K_t}\right]=\left\{0\right\}$, which proves the integrability of $K_t$, for every $t\in\R$.\\
Furthermore, for $t\neq0$, we get
\begin{eqnarray*}
H^{2\,-}_{K_t}\left(X;\R\right)\ni\left[e^{34}\right] &=&\left[e^{34}+\frac{1}{t}\,\de e^{3}\right] \;=\; \left[e^{34}+\frac{1}{t}\,(e^{23}+t\,e^{43}-t\,e^{43})\right]\\[5pt]
&=&\left[\frac{1}{t}\,(e^2-t\,e^4)\wedge e^3\right] \in H^{2\,+}_{K_t}\left(X;\R\right)
\end{eqnarray*}
and therefore we have that
$$ 0\;\neq\; \left[e^{34}\right] \;\in\; H^{2\,-}_{K_t}\left(X;\R\right)\,\cap\, H^{2\,+}_{K_t}\left(X;\R\right) \;.$$
In particular, for $t\neq0$, it follows that $K_t$ is not \Cp\ at the $2$-nd stage and hence it is not \Cf\ at the $2$-nd stage, as a consequence of Proposition \ref{prop:cpf-pf} (in fact, no invariant representative in the class $\left[e^{12}\right]=\left[e^1\wedge(e^2-te^4)+te^{14}\right]$ is of pure type with respect to $K_t$, the space of invariant $\de$-exact $2$-forms being $\R\left\langle (e^2-t\,e^4)\wedge e^3-t\,e^{34},\, (e^2-t\,e^4)\wedge e^4\right\rangle$). Therefore, for $t\neq 0$, we have
$$ \dim_\R H^{2\,+}_{K_t}(X;\R) \;=\; 1\;, \qquad \dim_\R H^{2\,-}_{K_t}(X;\R) \;=\; 1.$$
Note that, in this example, for every $t\in\R$ one has $s\left(\mathfrak{g}^-_{K_t}\right)=0$ and $s\left(\mathfrak{g}^+_{K_t}\right)=1$ but for $t\neq0$ the \para-complex structure $K_t$ is not \Cp\ at the $2$-nd stage, therefore the Abelian condition on just $\mathfrak{g}^-$ in Theorem \ref{thm 2.12} is not sufficient to have \Cp ness at the $2$-nd stage, as  announced in Remark \ref{rem:partial-abelian}.\\
Note moreover that, in this example, the functions
$$ \R\ni t\mapsto \dim_\R H^{2\,+}_{K_t}(X;\R)\in\N \qquad \text{ and } \qquad \R\ni t\mapsto \dim_\R H^{2\,-}_{K_t}(X;\R)\in\N $$
are, respectively, lower-semi-continuous and upper-semi-continuous.\\
Furthermore, we note that $X$ admits a symplectic form $\omega:=e^{12}+e^{34}$ which is compatible with the \para-complex structure $K_0$: therefore, $\left(X,\,K_0,\,\omega\right)$ is a \para-K\"ahler manifold. Instead, for $t\neq0$, one has $H^-_{K_t}\left(X;\R\right)=\R\left\langle e^{34}\right\rangle$ and therefore, if a $K_t$-compatible symplectic form $\omega_t$ existed, it should be in the same cohomology class as $e^{34}$ and then it should satisfy
$$ \Vol(X) \;=\; \int_X \omega_t\wedge\omega_t \;=\; \int_X e^{34}\wedge e^{34} \;=\; 0 \;, $$
which is not possible;
therefore, for $t\neq0$, there is
no symplectic structure compatible with the \para-complex structure $K_t$: in particular, $\left(X,\,K_t\right)$ admits no \para-\kal\ structure.
\end{ex}

The previous example proves the following result, giving a strong difference between the \para-complex and the complex cases (compare with the stability result of K\"ahlerness proved by K. Kodaira and D.~C. Spencer in \cite{kodaira-spencer-3}).
\begin{thm}\label{thm:para-kahler-deformations}
 The property of being \para-\kal\ is not stable under small deformations of the \para-complex structure.
\end{thm}

Furthermore, Example \ref{es 2.17} proves also the following instability result (a similar result holds also in the complex case, see \cite[Theorem 3.2]{angella-tomassini}).

\begin{prop}\label{prop:def}
 The property of being \Cp\ at the $2$-nd stage or \Cf\ at the $2$-nd stage is not stable under small deformations of the \para-complex structure.
\end{prop}

We recall that T. Dr\v{a}ghici, T.-J. Li and W. Zhang proved in \cite[Theorem 5.4]{draghici-li-zhang} that, given a curve of almost complex structures on a $4$-dimensional compact manifold, the dimension of the almost complex anti-invariant subgroup of the real second de Rham cohomology group is upper-semi-continuous and hence (as a consequence of \cite[Theorem 2.3]{draghici-li-zhang}) the dimension of the almost complex invariant subgroup of the real second de Rham cohomology group is lower-semi-continuous. This result holds no more true in dimension greater than $4$ (see \cite{angella-tomassini-2}).\\
We provide two examples showing that the dimensions of the \para-complex invariant and anti-invariant subgroups of the cohomology can jump along a curve of \para-complex structures.

\begin{ex}\label{ex:def-jump-sci}
{\itshape There exists a curve of \para-complex structures on a $6$-dimensional nilmanifold such that the dimensions of the \para-complex invariant and anti-invariant subgroups of the real second de Rham cohomology group jump (lo\-wer-se\-mi-con\-ti\-nuou\-sly) along the curve.}\\
Consider the $6$-dimensional nilmanifold
$$ X\;:=\; \left(0,\,0,\,0,\,12,\,13,\,24\right) \;.$$
By Nomizu's theorem (see \cite{nomizu}), it is straightforward to compute
$$ H^2_{dR}(X;\R) \;=\; \R\left\langle e^{14},\, e^{15},\, e^{23},\, e^{26},\, e^{35},\, e^{25}+e^{34} \right\rangle \;.$$
For every $t\in\left[0,\,1\right]$, consider the invariant \para-complex structure
$$
K_t \;:=\;
\left(
\begin{array}{cc|cc|cc}
 1 & & & & & \\
 & -1 & & & & \\
\hline
 & & \frac{(1-t)^2-t^2}{(1-t)^2+t^2} & \frac{2t(1-t)}{(1-t)^2+t^2} & & \\ 
 & & \frac{2t(1-t)}{(1-t)^2+t^2} & -\frac{(1-t)^2-t^2}{(1-t)^2+t^2} & & \\ 
\hline
 & & & & 1 & \\
 & & & & & -1
\end{array}
\right) \;.
$$
For $0\leq t\leq 1$, one checks that
$$ \mathfrak{g}^+_{K_t} \;=\; \R\left\langle e_1,\, (1-t)\,e_3+t\,e_4,\, e_5 \right\rangle \quad \text{ and } \quad \mathfrak{g}^-_{K_t} \;=\; \R\left\langle e_2,\, t\,e_3-(1-t)\,e_4,\, e_6 \right\rangle \;;$$
one can straightforwardly check that the integrability condition of $K_t$ is satisfied for every $t\in\left[0,\,1\right]$.\\
In particular, for $t\in\{0,\,1\}$, one has
$$ K_0 \;=\; \left( + \, - \, + \, - \, + \, - \right) \qquad \text{ and }\qquad K_1 \;=\; \left( + \, - \, - \, + \, + \, - \right) \;.$$
It is straightforward to check that $K_0$ and $K_1$ are \Cpf\ at the $2$-nd stage and
\begin{eqnarray*}
H^2_{dR}(X;\R) &=& \underbrace{\R\left\langle e^{15},\,e^{26},\,e^{35}\right\rangle}_{=\;H^{2\,+}_{K_0}(X;\R)}\,\oplus\,\underbrace{\R\left\langle e^{14},\,e^{23},\,e^{25}+e^{34} \right\rangle}_{=\;H^{2\,-}_{K_0}(X;\R)} \\[10pt]
 &=& \underbrace{\R\left\langle e^{14},\,e^{15},\,e^{23},\,e^{26}\right\rangle}_{=\;H^{2\,+}_{K_1}(X;\R)}\,\oplus\,\underbrace{\R\left\langle e^{35},\,e^{25}+e^{34} \right\rangle}_{=\;H^{2\,-}_{K_1}(X;\R)} \;;
\end{eqnarray*}
therefore
$$ \dim_\R H^{2\,+}_{K_0}(X;\R) \;=\; 3 \qquad \text{ and } \qquad \dim H^{2\,-}_{K_0}(X;\R) \;=\; 3 $$
and
$$ \dim_\R H^{2\,+}_{K_1}(X;\R) \;=\; 4 \qquad \text{ and } \qquad \dim H^{2\,-}_{K_1}(X;\R) \;=\; 2 \;. $$
Instead, for $0<t<1$, one has
$$ H^{2\,+}_{K_t}(X;\R) \;=\; \R\left\langle e^{14},\, e^{15},\, e^{23},\, e^{26}\right\rangle $$
and
$$ H^{2\,-}_{K_t}(X;\R) \;=\; \R\left\langle e^{14},\, e^{23},\, e^{25}+e^{34}\right\rangle \;; $$
it follows that, for $0<t<1$, the \para-complex structure $K_t$ is neither \Cp\ at the $2$-nd stage nor \Cf\ at the $2$-nd stage; moreover, for $0<t<1$, one gets
$$ \dim_\R H^{2\,+}_{K_t}(X;\R) \;=\; 4 \qquad \text{ and } \qquad \dim H^{2\,-}_{K_t}(X;\R) \;=\; 3 \;: $$
in particular, the functions
$$ \left[0,\,1\right] \ni t \mapsto \dim_\R H^{2\,+}_{K_t}(X;\R) \in \N  \qquad \text{ and } \qquad \left[0,\,1\right] \ni t \mapsto \dim_\R H^{2\,-}_{K_t}(X;\R) \in \N  $$
are non-constant and both lower-semi-continuous.
\end{ex}

The previous examples show that the dimension of the \para-complex anti-invariant subgroup of the de Rham cohomology in general is not upper-semi-continuous (it is such in Example \ref{es 2.17}) or lower-semi-continuous (it is such in Example \ref{ex:def-jump-sci}). We end this section with an example showing that also the dimension of the \para-complex invariant subgroup of the de Rham cohomology in general is not lower-semi-continuous (it is such in Example \ref{es 2.17} and in Example \ref{ex:def-jump-sci}).

\begin{ex}\label{ex:def-jump-scs}
{\itshape There exists a curve of \para-complex structures on a $6$-dimensional nilmanifold such that the dimensions of the \para-complex invariant and anti-invariant subgroups of the real second de Rham cohomology group jump (up\-per-se\-mi-con\-ti\-nuou\-sly) along the curve.}\\
Consider the $6$-dimensional nilmanifold
$$ X\;:=\; \left(0,\,0,\,0,\,12,\,13,\,24\right) \;.$$
By Nomizu's theorem (see \cite{nomizu}), it is straightforward to compute
$$ H^2_{dR}(X;\R) \;=\; \R\left\langle e^{14},\, e^{15},\, e^{23},\, e^{26},\, e^{35},\, e^{25}+e^{34} \right\rangle \;.$$
For every $t\in\left[0,\,1\right]$, consider the invariant \para-complex structure
$$
K_t \;:=\;
\left(
\begin{array}{cccc|cc}
 1 & & & & & \\
 & -1 & & & & \\
 & & -1 & & & \\
 & & & 1 & & \\
 \hline
 & & & & \frac{(1-t)^2-t^2}{(1-t)^2+t^2} & \frac{2t(1-t)}{(1-t)^2+t^2} \\
 & & & & \frac{2t(1-t)}{(1-t)^2+t^2} & -\frac{(1-t)^2-t^2}{(1-t)^2+t^2} \\
\end{array}
\right) \;.
$$
For $0\leq t\leq 1$, one checks that
$$ \mathfrak{g}^+_{K_t} \;=\; \R\left\langle e_1,\, e_4,\, (1-t)\,e_5+t\, e_6 \right\rangle \quad \text{ and } \quad \mathfrak{g}^-_{K_t} \;=\; \R\left\langle e_2,\, e_3,\, t\, e_5-(1-t)\, e_6 \right\rangle \;;$$
one can straightforwardly check that the integrability condition of $K_t$ is satisfied for every $t\in\left[0,\,1\right]$. Furthermore, one can prove that $K_t$ is Abelian for every $t\in\left[0,\,1\right]$, hence it is in particular \Cp\ at the $2$-nd stage by Corollary \ref{cor:abelian}.\\
In particular, for $t\in\{0,\,1\}$, one has
$$ K_0 \;=\; \left( + \, - \, - \, + \, + \, - \right) \qquad \text{ and }\qquad K_1 \;=\; \left( + \, - \, - \, + \, - \, + \right) \;.$$
It is straightforward to check that $K_0$ is \Cpf\ at the $2$-nd stage and
$$ H^2_{dR}(X;\R) \;=\; \underbrace{\R\left\langle e^{14},\,e^{15},\,e^{23},\, e^{26}\right\rangle}_{=\;H^{2\,+}_{K_0}(X;\R)}\,\oplus\,\underbrace{\R\left\langle e^{35},\,e^{25}+e^{34} \right\rangle}_{=\;H^{2\,-}_{K_0}(X;\R)} $$
while $K_1$ is \Cp\ at the $2$-nd stage, non-\Cf\ at the $2$-nd stage and
$$ H^2_{dR}(X;\R) \;=\; \underbrace{\R\left\langle e^{14},\,e^{23},\,e^{35} \right\rangle}_{=\;H^{2\,+}_{K_1}(X;\R)}\,\oplus\,\underbrace{\R\left\langle e^{15},\,e^{26} \right\rangle}_{=\;H^{2\,-}_{K_1}(X;\R)} \,\oplus\, \R\left\langle e^{25}+e^{34}\right\rangle \;,$$
where
$$ \R\left\langle e^{25}+e^{34}\right\rangle \,\cap\,\left(H^{2\,+}_{K_1}(X;\R)\oplus H^{2\,-}_{K_1}(X;\R)\right)\;=\;\left\{0\right\} \;; $$
therefore
$$ \dim_\R H^{2\,+}_{K_0}(X;\R) \;=\; 4 \qquad \text{ and } \qquad \dim H^{2\,-}_{K_0}(X;\R) \;=\; 2 $$
and
$$ \dim_\R H^{2\,+}_{K_1}(X;\R) \;=\; 3 \qquad \text{ and } \qquad \dim H^{2\,-}_{K_1}(X;\R) \;=\; 2 \;. $$
Instead, for $0<t<1$, one has
$$ H^{2\,+}_{K_t}(X;\R) \;=\; \R\left\langle e^{14},\, e^{23} \right\rangle $$
and
$$ H^{2\,-}_{K_t}(X;\R) \;=\; \R\left\langle t\, e^{26}+(1-t)\, e^{25}+(1-t)\, e^{34} \right\rangle \;, $$
while
$$ \R\left\langle e^{15},\, e^{35},\, e^{26} \right\rangle \,\cap\, \left(H^{2\,+}_{K_t}(X;\R)\oplus H^{2\,-}_{K_t}(X;\R)\right)\;=\;\left\{0\right\} \;; $$
it follows that, for $0<t<1$, the \para-complex structure $K_t$ is \Cp\ at the $2$-nd stage and non-\Cf\ at the $2$-nd stage; moreover, for $0<t<1$, one gets
$$ \dim_\R H^{2\,+}_{K_t}(X;\R) \;=\; 2 \qquad \text{ and } \qquad \dim H^{2\,-}_{K_t}(X;\R) \;=\; 1 \;: $$
in particular, the functions
$$ \left[0,\,1\right] \ni t \mapsto \dim_\R H^{2\,+}_{K_t}(X;\R) \in \N  \qquad \text{ and } \qquad \left[0,\,1\right] \ni t \mapsto \dim_\R H^{2\,-}_{K_t}(X;\R) \in \N  $$
are non-constant and both upper-semi-continuous.
\end{ex}

We resume the contents of Example \ref{ex:def-jump-sci} and Example \ref{ex:def-jump-scs} in the following proposition.

\begin{prop}
 Let $X$ be a compact manifold and let $\left\{K_t\right\}_{t\in I}$ be a curve of \para-complex structures on $X$, where $I\subseteq\R$. Then, in general, the functions
$$ I\ni t \mapsto \dim_\R H^{2\,+}_{K_t}(X;\R)\in\N \qquad \text{ and } \qquad  I\ni t \mapsto \dim_\R H^{2\,-}_{K_t}(X;\R)\in\N $$
are not upper-semi-continuous or lower-semi-continuous.
\end{prop}

\end{document}